\documentclass[12pt]{amsart}
\usepackage{color}
\usepackage{amsmath}
\usepackage{amsthm,amssymb}
\usepackage{enumerate}
\theoremstyle{plain}
\newtheorem{theorem}{Theorem}[section]
\newtheorem{lemma}[theorem]{Lemma}

\theoremstyle{definition}

\newcommand{\sign}{\mbox{\rm sign}}

\usepackage[linktocpage]{hyperref} 
\usepackage{bookmark} 
\hypersetup{
colorlinks=true, 
linkcolor=blue, 
citecolor=red, 
filecolor=magenta, 
urlcolor=blue,
linktoc=all
}
\usepackage{tikz,xcolor,hyperref}
\usepackage{soul}
\begin{document}
\title[Maximal subfields in division algebras]{Maximal subfields in division algebras generated by images of polynomials}
\author[L.Q. Danh]{Le Qui Danh\textsuperscript{1\dag*}}
\author[T. T. Deo]{Trinh Thanh Deo\textsuperscript{2,3\S}} 
\keywords{division ring; maximal subfield; multilinear polynomial; word\\ \protect \indent * Corresponding author: L.Q. Danh}
\thanks{This research was funded by Vietnam National Foundation for Science and Technology Development (NAFOSTED) under Grant No. 101.04-2023.18}

\maketitle

\begin{center}
{\small 

\textsuperscript{1}Faculty of Fundamental Sciences,\\ University of Architecture Ho Chi Minh City,\\ 196 Pasteur Str., Dist. 3, Ho Chi Minh City, Vietnam \\
\textsuperscript{2}Faculty of Mathematics and Computer Science, University of Science, Ho Chi Minh City, Vietnam\\
\textsuperscript{3} Vietnam National University,\\ Ho Chi Minh City, Vietnam\\
\textsuperscript{\dag}\url{danh.lequi@uah.edu.vn}\\ 
\textsuperscript{\S}\url{ttdeo@hcmus.edu.vn}
}
\end{center}

\begin{abstract}
Let $D$ be a division ring with center $F$, $f(x_1,x_2,\dots, x_m)$ a non-central multilinear polynomial over $F$, and  $w(x_1,x_2,\dots,x_m)$ a non-trivial word. In this paper, we investigate conditions under which there exists an element $a \in D$ such that the subfield $F(a)$ generated by $a$ is a maximal subfield of $D$. Specifically, we prove that there always exists an element $a$ in the set
\[
\{f(a_1,\dots,a_m)\mid a_1,\dots, a_m\in D \} \cup \{w(a_1,\dots,a_m)\mid a_1,\dots, a_m\in D \backslash \{0\} \}
\]
such that $F(a)$ is a maximal subfield of $D$. This result shows that maximal subfields can be generated by evaluating polynomial or group word expressions at elements of $D$.
\end{abstract}
\maketitle

\section{Introduction}

\subsection{An open problem and the aim of the paper} $\nonumber$

Let $D$ be a division ring with center $F$. A subfield of $D$ is said to be \textit{maximal} if it is not properly contained in any other subfield of $D$. In this paper, we consider the following natural question.

\noindent\textbf{Problem.} Let $D$ be a finite-dimensional division ring over its center $F$. What conditions must an element $a \in D$ satisfy so that the subfield $F(a)$ generated by $a$ over $F$ is a maximal subfield of $D$?

This question arises from problems 28 and 29 in the paper \cite{Pa_Ma_00} by M. Mahdavi-Hezavehi. Some special cases related to this question have been studied and results have been obtained. For instance, in \cite{Aaghabali.Bien.2019paper}, the authors proved that if $a$ is of the form $xyx^{-1}y^{-1}$ or $xy-yx$ for some $x$ and $y$ in $D^*$, then the subfield $F(a)$ is maximal in $D$. More general,  the conclusion of this result holds when considering $x$ or $y$ in a subnormal subgroup of $D^*$ (see \cite{Pa_Bien.Hai.Trang_2021}). In this paper, we prove that one can choose an element $a$ of the form
\[
a = f(a_1, a_2, \dots, a_m) \quad \text{with } a_1, a_2, \dots, a_m \in D \setminus \{0\},
\]
where $f(x_1, x_2, \dots, x_m)$ is a polynomial over $F$ of one of the following two types:
\begin{itemize}
\item The first type is a \textit{non-central multilinear polynomial}:
\[
f(x_1, x_2, \dots, x_m) = \sum_{\sigma \in S_m} a_\sigma x_{\sigma(1)} x_{\sigma(2)} \dots x_{\sigma(m)},
\]
where $S_m$ is the symmetric group on $m$ letters, and $a_\sigma \in F$.
\item The second type is a \textit{group monomial}:
\[
f(x_1, x_2, \dots, x_m) =x_{i_1}^{n_1} x_{i_2}^{n_2} \dots x_{i_t}^{n_t}, \quad \text{with } 1 \le i_j \le m \text{ and } n_j \in \mathbb{Z}.
\]
\end{itemize}

Our approach relies on a blend of classical techniques and recent developments concerning the images of Laurent polynomials. While the tools may seem familiar, they lead to new structural insights into maximal subfields of division rings.

\subsection{Motivation and structure of the paper}$\nonumber$

One common approach to studying the structure of a division ring is to investigate its maximal subfields. A natural question is whether the properties of such maximal subfields can reflect the overall structure of the ring. This approach has been taken in several works, such as \cite{Pa_HaTrBi_19, Pa_BeDrShSh_13, Pa_AaAkBi_18, Pa_Thu}, where the authors examine algebraic properties of division rings through their subfields.

One particular aspect of this approach is the study of maximal subfields in division rings. Maximal subfields play a crucial role in many classical theorems. A prominent example is the Brauer-Albert theorem: If $D$ is a division ring with $\mathrm{dim}_F D = n^2$, and if $K = F(a)$ is a simple maximal subfield, then there exists $b \in D$ such that the set $\{ a^i b a^j \mid 1 \le i, j \le n \}$ forms a basis for $D$ over $F$ (see \cite[Theorem 15.16]{Lam.2001.book}). This shows that elements generating maximal subfields contain rich structural information about $D$, and identifying or describing such elements is a non-trivial task.

Moreover, simple subfields have been extensively studied in the theory of division rings. A series of works by M. Mahdavi-Hezavehi has focused on this topic, and his paper \cite{Pa_Ma_00} offers valuable insights from various perspectives. The two open problems 28 and 29 in that paper directly inspired the line of investigation pursued in the present work.

The paper is organized as follows:  
\begin{itemize}
    \item In Section 2, we present some background on the theory of identities, which is essential since our main results are based on this theory.. A solid understanding of these basic tools will make it easy for the reader to engage	 with the later sections.
    
    \item In Section 3, we provide auxiliary results on the image of Laurent polynomials, which are used to establish the existence of algebraic matrices of prescribed degrees.  
    \item Finally, Section 4 contains the main results of the paper, where we give and prove sufficient criteria for an element to generate a maximal subfield in a division ring and we then present an application of these results.
\end{itemize}

\section{Preliminaries}

\subsection{Laurent polynomial identities}~

In this paper, the main results are proved by applying some results on Laurent polynomial identities. Let $F$ be a field, and $X = \{x_1, x_2, \dots, x_m\}$ be a set of $m$ noncommuting indeterminates. Denote by $F[X]$ the free algebra on $X$ over $F$, and by $F\langle X \rangle$ the group algebra of $\langle X\rangle$ over $F$, where $\langle X\rangle$ is the free group generated by $X$.  We can see $F[X]$ as a subalgebra of $F\langle X\rangle$. Elements in $F[X]$ are called (noncommutative) \textit{polynomials}, and in $F\langle X\rangle$ are called (noncommutative) \textit{Laurent polynomials}. An element $f(X)$ in $F\langle X \rangle$ has the general form
\[
f(X) = \sum_{i=1}^n c_i x_{i_1}^{n_{i_1}} x_{i_2}^{n_{i_2}} \dots x_{i_t}^{n_{i_t}},
\]
where $c_i \in F$, $x_{i_j} \in X$, and $n_{i_j} \in \mathbb{Z}$. If all exponents $n_{i_j}$ are non-negative, then $f(X)$ is a polynomial in $F[X]$.

Let $R$ be an algebra over $F$ with the multiplicative group $R^*$, and $f(X)$ be a Laurent polynomial in $F\langle X\rangle$. For all $a_1, a_2,\dots, a_m \in R^*$ we denote by $f(a_1,a_2,\ldots,a_m)$ the value of $f$ by substituting $x_i$ with $a_i$, that is, if $f(X) = \sum\limits_{i=1}^n c_i x_{i_1}^{n_{i_1}} x_{i_2}^{n_{i_2}} \dots x_{i_t}^{n_{i_t}}$, then $f(a_1,a_2,\ldots,a_m)=\sum\limits_{i=1}^n c_i a_{i_1}^{n_{i_1}} a_{i_2}^{n_{i_2}} \dots a_{i_t}^{n_{i_t}}$. 

The Laurent polynomial $f(X)$ is said to be a \emph{Laurent polynomial identity} of $R$ if $f(a_1, a_2,\dots, a_m) = 0$ for all $a_1, a_2, \dots, a_m \in R^*$. In this case, we also say that $R$ \emph{satisfies} the Laurent identity $f = 0$. The set of all Laurent polynomial identities of $R$ is denoted by $\mathcal{I}(R)$.

One of the key tools in this paper is the following consequence of a classical result:

\begin{lemma}\label{l2.2}
Let $F$ be an infinite field and $D$ a division ring with center $F$ such that $\dim_F D = n^2$. If $L$ is a field extension of $F$, then
\[
\mathcal{I}(D) = \mathcal{I}(\mathrm{M}_n(F)) = \mathcal{I}(\mathrm{M}_n(L)).
\]
\end{lemma}

\begin{proof}
This is a special case of \cite[Theorem 11]{Amitsur.1966paper}.
\end{proof}

\subsection{Algebraic elements of bounded degree}$\nonumber$

Let $F$ be a field and $R$ be a ring whose center contains $F$. An element $a \in R$ is said to be \emph{algebraic over $F$} if there exists a nonzero polynomial $f(x) \in F[x]$ such that $f(a) = 0$. Furthermore, $a$ is called \emph{algebraic of degree $n$ over $F$} if there exists such a polynomial $f(x)$ of degree $n$, and there is no nonzero polynomial of smaller degree satisfying $f(a) = 0$.

It is well known that every element in $R = \mathrm{M}_n(F)$ is algebraic over $F$, since each matrix satisfies its own characteristic polynomial. However, not every matrix has degree exactly $n$, although matrices with this property do exist. This observation plays an essential role in our techniques, especially when working with a particular class of Laurent polynomials introduced below.

Let $n$ be a positive integer, and let $x, x_1, \dots, x_n$ be noncommuting indeterminates. Define:
\[
g_n(x, x_1, \ldots, x_n) = \sum_{\sigma \in S_{n+1}} \sign(\sigma) . x^{\sigma(0)} x_1 x^{\sigma(1)} x_2 \dots x_n x^{\sigma(n)},
\]
where $S_{n+1}$ denotes the symmetric group on $\{0, 1, \dots, n\}$, and $\sign(\sigma)$ denotes the sign of the permutation $\sigma$. This rational expression, introduced in \cite{Beidar.Martindale.Mikhalev}, relates algebraic elements of bounded degree to noncommutative polynomials in \mbox{$n+1$} indeterminates. It may be viewed as a generalization of the characteristic polynomial in the setting of noncommutative algebras.

\begin{lemma}\label{le2.2}
Let $D$ be a division ring with center $F$, and $r$ be a positive integer. For any element $a \in \mathrm{M}_r(D)$, the following conditions are equivalent.
\begin{enumerate}
\item[{\rm(1)}] The element $a$ is algebraic over $F$ of degree at most $n$.
\item[{\rm(2)}] For all $r_1, r_2, \dots, r_n \in R$, we have $g_n(a, r_1, r_2, \dots, r_n) = 0$.
\end{enumerate}
\end{lemma}

\begin{proof}
This follows directly from \cite[Corollary 2.3.8]{Beidar.Martindale.Mikhalev}.
\end{proof}

We conclude this section with a lemma ensuring that applying $g_n$ to a nonzero Laurent polynomial yields a nonzero Laurent polynomial:

\begin{lemma}\label{l2.3}
Let $D$ be a division ring with center $F$, $g_n(x, x_1, \dots, x_n)$ be defined as above, and $f(y_1, y_2, \dots, y_m)$ be a nonzero Laurent polynomial in $F\langle y_1, y_2, \dots, y_m \rangle$. Then the expression
\[
g_n(f(y_1, \dots, y_m), x_1, \dots, x_n)
\]
is a nonzero Laurent polynomial in $F\langle x_1, \dots, x_n, y_1, \dots, y_m \rangle$.
\end{lemma}

\begin{proof} 
    As a corollary of \cite[(14.21)]{Lam.2001.book}, there is a division ring $D_1$ with center $F$ such that $D$ is a subring of $D_1$ and $\mathrm{dim}_F D_1$ is infinite. If for all $a_1,a_2,\ldots,a_m$ in $D_1$, $f(a_1,a_2,\ldots,a_m)$ is algebraic of degree at most $n$ over $F$, then
    \cite[Theorem 1.2]{Pa_Hai.Dung.Bien_2022} follows that $D_1$ is finite-dimensional over $F$. That is a contradiction. Thus,there exist $a_1,a_2,\ldots,a_m$ in $D_1$ such that $f(a_1,a_2,\ldots,a_m)$ is either not algebraic or algebraic of degree greater than $n$ over $F$. By Lemma \ref{le2.2},
    \[
    g_n(f(a_1,a_2,\ldots,a_m),r_1,r_2,\ldots,r_n)\neq 0
    \]
    for some $r_1,r_2,\ldots,r_n\in D_1$. Consequently, $g_n(f(y_1, \dots, y_m), x_1, \dots, x_n)$ is nonzero.
\end{proof}

\section{Some matrices in the image of a polynomial}

The goal of this section is to establish the existence of matrices in the image of a given polynomial, under the condition that such matrices have a prescribed size and are algebraic. This result plays a crucial role in the overall structure of the paper.

Recall that the trace of a matrix $A\in \mathrm{M}_n(F)$, denoted by $\mathrm{trace}(A)$, is the sum of its diagonal entries.
Next, we state an essential result concerning multilinear polynomials on $2\times2$ matrices, due to Kanel-Belov, Malev, and Rowen:

\begin{lemma}\label{caseOfThm2_Kanel-Belov.Malev.Rowen.2012paper}
Let $F$ be an algebraically closed field, and $f(x_1, \ldots, x_n)$ a multilinear polynomial in $\mathrm{M}_2(F)$. If $f$ is non-central over $\mathrm{M}_2(F)$, then for every matrix $A \in \mathrm{M}_2(F)$ with $\mathrm{trace}(A) = 0$, there exist matrices $T_1, \ldots, T_n \in \mathrm{M}_2(F)$ such that
$f(T_1, \ldots, T_n) = A.$
\end{lemma}

\begin{proof}
According to \cite[Theorem 2]{Kanel-Belov.Malev.Rowen.2012paper}, if $f(\mathrm{M}_2(F)) \not\subseteq \mathrm{Z}(\mathrm{M}_2(F)) \cong F$, then either $f(\mathrm{M}_2(F)) = \mathrm{M}_2(F)$ or $f(\mathrm{M}_2(F))$ is the set of trace-zero matrices. Thus, for any $A \in \mathrm{M}_2(F)$ with $\mathrm{trace}(A) = 0$, there exist matrices  $T_1, \ldots, T_n \in \mathrm{M}_2(F)$ such that $f(T_1, \ldots, T_n) = A$.
\end{proof}

Throughout the remainder of this section, for each integer $m \geq 2$, we fix two matrices $P_m$ and $Q_m$ over an infinite field $F$ as follows. Let $s$ be the integer part of $m/2$. Choose elements $a_1, \ldots, a_s, b_1, \ldots, b_s \in F \setminus \{0,1\}$ such that the elements $\pm a_1, \ldots, \pm a_s$ are pairwise distinct, and the elements $b_1, b_1^{-1}, \ldots, b_s, b_s^{-1}$ are also pairwise distinct.

For each $i = 1, \ldots, s$, define the following $2 \times 2$ matrices
\[
A_i = \begin{pmatrix} a_i & 1 \\ 0 & -a_i \end{pmatrix}, \text{ and } 
B_i = \begin{pmatrix} b_i & 0 \\ 0 & b_i^{-1} \end{pmatrix}.
\]

We then define
\[
P_m =
\begin{cases}
\mathrm{diag}(A_1, \ldots, A_s), & \text{if } m = 2s, \\
\mathrm{diag}(A_1, \ldots, A_s, 0), & \text{if } m = 2s+1,
\end{cases}\]
and
\[
Q_m =
\begin{cases}
\mathrm{diag}(B_1, \ldots, B_s), & \text{if } m = 2s, \\
\mathrm{diag}(B_1, \ldots, B_s, 1), & \text{if } m = 2s+1,
\end{cases}
\]
where $\mathrm{diag}(X_1, \ldots, X_n)$ represents a block diagonal matrix with the square matrices $X_1,\ldots, X_n$ along the diagonal of a larger matrix, all other elements being zeros.

We now develop several auxiliary lemmas to analyze the properties of the matrices $P_m$ and $Q_m$. The following lemma guarantees that $P_m$ can be obtained as the value of $f$ and also establishes its algebraicity.

\begin{lemma}\label{Lem.Pm}
Let $F$ be an algebraically closed field, and $f(x_1, \ldots, x_n)$ a non-central multilinear polynomial over $\mathrm{M}_2(F)$. For the matrix $P_m$ defined above, the following assertions hold:
\begin{enumerate}
\item[{\rm (1)}] There exist matrices $T_1, \ldots, T_n \in \mathrm{M}_m(F)$ such that $f(T_1, \ldots, T_n) = P_m$.
\item[{\rm(2)}] $P_m$ is an algebraic matrix of degree $m$ over $F$.
\end{enumerate}
\end{lemma}

\begin{proof}
(1) For each $i$, we have $\mathrm{trace}(A_i) = 0$. Since $f$ is non-central, Lemma~\ref{caseOfThm2_Kanel-Belov.Malev.Rowen.2012paper} implies that there exist matrices $T_{i1}, \ldots, T_{in} \in \mathrm{M}_2(F)$ such that
\[
f(T_{i1}, \ldots, T_{in}) = A_i.
\]
For each $j = 1, \ldots, n$, define:
\[
T_j =
\begin{cases}
\mathrm{diag}(T_{1j}, \ldots, T_{sj}), & \text{if } m = 2s, \\
\mathrm{diag}(T_{1j}, \ldots, T_{sj}, 0), & \text{if } m = 2s+1.
\end{cases}
\]
Then we obtain $f(T_1, \ldots, T_n) = P_m$.

(2) Consider the characteristic polynomial $\chi_{P_m}(x)$ of $P_m$.

\emph{Case 1.} If $\mathrm{char}(F) \neq 2$, then
\[
\chi_{P_m}(x) = x^{\delta_{m,2s+1}} \prod_{i=1}^s (x - a_i)(x + a_i).
\]
Here $\delta_{m,2s+1}$ is the Kronecker delta. Since the roots are pairwise distinct, this is also the minimal polynomial of $P_m$, and hence $P_m$ is algebraic of degree $m$.

\emph{Case 2.} If $\mathrm{char}(F) = 2$, then $\chi_{A_i}(x) = (x + a_i)^2$ and $A_i + a_iI \neq 0$, so the minimal polynomial of $A_i$ is $(x + a_i)^2$. Thus,
\[
\chi_{P_m}(x) = x^{\delta_{m,2s+1}} \prod_{i=1}^s (x + a_i)^2,
\]
and therefore $P_m$ is also algebraic of degree $m$.
\end{proof}

We also need a result on the image of word maps in the group $\mathrm{SL}_2(F)$ in \cite[Theorem 1.4]{Bien.Ramezan-Nassab.Trung.2024paper}.

\begin{lemma}[\protect{\cite[Theorem 1.4]{Bien.Ramezan-Nassab.Trung.2024paper}}]\label{thm1.4_Bien.Ramezan-Nassab.Trung.2024paper}
Let $F$ be an algebraically closed field, and $w(x_1, \ldots, x_n)$ a non-trivial word of non-commuting indeterminants $x_1,\ldots, x_n$. Then, the image of the corresponding map $w: \prod\limits_n \mathrm{SL}_2(F) \to \mathrm{SL}_2(F)$ contains all matrices in $\mathrm{SL}_2(F)$ whose trace is different from $\pm 2$.
\end{lemma}

Finally, for the block-diagonal matrix $Q_m$, we have the following property.

\begin{lemma}\label{Lem.Qm}
Let $F$ be an algebraically closed field, and $w(x_1, \ldots, x_n)$ a non-trivial word. For the matrix $Q_m$ defined above, the following assertions hold.
\begin{enumerate}
\item[{\rm (1)}] There exist matrices $S_1, \ldots, S_n \in \mathrm{SL}_m(F)$ such that $w(S_1, \ldots, S_n) = Q_m$.
\item[{\rm(2)}] $Q_m$ is an algebraic matrix of degree $m$ over $F$.
\end{enumerate}
\begin{proof}
(1) For each $i$, since $B_i \in \mathrm{SL}_2(F)$ and $\mathrm{trace}(B_i) \neq \pm 2$, Lemma~\ref{thm1.4_Bien.Ramezan-Nassab.Trung.2024paper} implies that there exist matrices $S_{i1}, \ldots, S_{in} \in \mathrm{SL}_2(F)$ such that
\[
w(S_{i1}, \ldots, S_{in}) = B_i.
\]
For each $j = 1, \ldots, n$, define:
\[
S_j =
\begin{cases}
\mathrm{diag}(S_{1j}, \ldots, S_{sj}), & \text{if } m = 2s, \\
\mathrm{diag}(S_{1j}, \ldots, S_{sj}, 1), & \text{if } m = 2s+1.
\end{cases}
\]
Then $w(S_1, \ldots, S_n) = Q_m$.

(2) The characteristic polynomial $\chi_{Q_m}(x)$ of $Q_m$ is given by:
\[
\chi_{Q_m}(x) = (x - 1)^{\delta_{m,2s+1}} \prod_{i=1}^s (x - b_i)(x - b_i^{-1}).
\]
Since the roots of $\chi_{Q_m}(x)$ are pairwise distinct, this polynomial is also the minimal polynomial of $Q_m$, and thus $Q_m$ is algebraic of degree $m$.
\end{proof}
\end{lemma}

\section{Main Results}

In this section, we present two main results concerning the construction of maximal subfields in a finite-dimensional division ring, via non-central multilinear polynomials and non-trivial words and present an application.

\subsection{Simple maximal subfields}$\nonumber$

 We begin by recalling a key lemma that plays a significant role throughout the argument.

\begin{lemma}[\protect{\cite[Lemma 3.3]{Aaghabali.Bien.2019paper}}]\label{lem3.3_Aaghabali.Bien.2019paper}
Let $D$ be a division ring with center $F$, and $K$ a subfield of $D$ containing $F$. If $\mathrm{dim}_F(D)=m^2$, then $\mathrm{dim}_F(K)\leq m$. Equality holds if and only if $K$ is a maximal subfield of $D$.
\end{lemma}

This result shows that in a finite-dimensional division ring $D$ over a field $F$, maximal subfields have dimension exactly equal to the square root of the dimension of $D$ over $F$. Therefore, finding an element $a \in D$ such that $F(a)$ has dimension $m$ yields a maximal subfield of $D$.

Next, we prove that the values obtained from a non-central multilinear polynomial can generate a maximal subfield, provided appropriate elements are substituted into the indeterminates.

\begin{theorem}\label{MaxSbFi_poly}
Suppose $D$ is a finite-dimensional division ring over its center $F$. Let $f(x_1,\ldots,x_n)\in F[x_1,\ldots,x_n]$ be a multilinear polynomial such that $f$ is non-central over the matrix ring $\mathrm{M}_2(F)$. Then, there exist elements $c_1,\ldots,c_n\in D$ such that the field $F(f(c_1,\ldots,c_n))$ is a maximal subfield of $D$.
\end{theorem}

\begin{proof}
Assume $\mathrm{dim}_F(D)=m^2$. Consider
\[
\ell=\max\left\{\mathrm{dim}_F(F(f(c_1,\ldots,c_n))) \,\middle|\, c_1,\ldots,c_n\in D\right\}.
\]
By Lemma~\ref{le2.2},
\[
g_\ell(f(c_1,\ldots,c_n), r_1,\ldots, r_\ell)=0
\]
for all $r_1,\ldots,r_\ell\in D$. Hence, $D$ satisfies the rational identity:
\[
g_\ell(f(x_1,\ldots,x_n), y_1,\ldots,y_\ell) = 0.
\]
By Lemma~\ref{l2.3}, this Laurent polynomial is nonzero, and thus, by Lemma~\ref{le2.2}, it is also an identity of the ring $\mathrm{M}_m(\overline{F})$, where $\overline{F}$ is the algebraic closure of $F$.  Therefore,
\[
g_\ell(f(X_1,\ldots,X_n), Y_1,\ldots,Y_\ell) = 0
\]
for all $X_i \in \mathrm{GL}_m(\overline{F})$ and $Y_j \in \mathrm{M}_m(\overline{F})$. Lemma~\ref{le2.2} implies that the matrices $f(X_1,\ldots,X_n)$ are algebraic over $\overline{F}$ of degree at most $\ell$.

Moreover, there exists a matrix $P_m$ (see details in Lemma~\ref{Lem.Pm}) that is algebraic of exact degree $m$ over $F$ and can be expressed as $P_m = f(T_1,\ldots,T_n)$ for some matrices $T_i$. This implies that $\ell \geq m$.

On the other hand, by Lemma~\ref{lem3.3_Aaghabali.Bien.2019paper}, every subfield $K$ of $D$ containing $F$ has dimension at most $m$, hence $\ell \leq m$. Therefore:
\[
\ell = m = \max\left\{\mathrm{dim}_F(F(f(c_1,\ldots,c_n))) \,\middle|\, c_1,\ldots,c_n \in D\right\}.
\]
Thus, there exist elements $c_1,\ldots,c_n \in D$ such that $\mathrm{dim}_F(F(f(c_1,\ldots,c_n))) = m$, which means $F(f(c_1,\ldots,c_n))$ is a maximal subfield of~$D$.
\end{proof}

The next result shows that a similar conclusion holds when replacing the polynomial $f$ by a non-central word, i.e., an algebraic expression obtained from products of indeterminates and constants.

\begin{theorem}\label{MaxSbFi_word}
Suppose $D$ is a finite-dimensional division ring over its center $F$.  
If $w(x_1,\ldots,x_n)$ is a non-trivial word, then there exist elements $c_1,\ldots,c_n \in D$ such that $F(w(c_1,\ldots,c_n))$ is a maximal subfield of~$D$.
\end{theorem}

\begin{proof}
The proof follows the same argument as in Proposition~\ref{MaxSbFi_poly}, replacing $f$ with $w$ and using a suitable matrix $Q_m$ corresponding to the word case.
\end{proof}

\subsection{An application}$\nonumber$

To conclude this paper, we present an application of the results just established.  A classical theorem of Jacobson asserts that if a division ring \(D\) with center \(F\) satisfies the property that every element is algebraic of bounded degree over \(F\), then \(D\) is finite‑dimensional over \(F\) (see \cite[Theorem 7]{Pa_Ja_45}).  Moreover, one obtains an explicit bound: if there exists a positive integer \(d\) such that every element of \(D\) is algebraic of degree at most \(d\) over \(F\), then
$\dim_F D \;\le\; d^2$,
as shown in \cite[Theorem 6]{Pa_ChFoLe_04}.  The method employed in \cite{Pa_ChFoLe_04} is precisely the technique introduced in Section 2.  In fact, several recent works (for example \cite{Pa_AaAkBi_18, Trang.Bien.Dung.Hai.2022paper}) also exploit this approach.  By combining the results proved in this section with the machinery of Section 2, we obtain the following extension of \cite[Theorem 6]{Pa_ChFoLe_04} and of the main theorem of \cite{Trang.Bien.Dung.Hai.2022paper}.  To state it, we first require the following auxiliary lemma.

\begin{lemma}\label{lem:auxiliary}
Let \(D\) be a division ring whose center \(F\) is infinite, and let $f(x_1,\dots,x_m)$ be either a non‑central multilinear polynomial over \(F\) or a non‑trivial group word.  If for every choice \(a_1,\dots,a_m\in D^*\) the value \(f(a_1,\dots,a_m)\) is algebraic over \(F\) of bounded degree, then \(D\) is finite‑dimensional over \(F\).
\end{lemma}

\begin{proof}
This is a special case of \cite[Theorem 1.2]{Pa_Hai.Dung.Bien_2022}.
\end{proof}

\begin{theorem}\label{thm:dimension_bound}
Let \(D\) be a division ring with infinite center \(F\), and $f(x_1,\ldots,x_m)$ be either a non‑central multilinear polynomial over $F$ or a non‑trivial group word.  Assume there is an integer \(d>0\) such that for every $a_1,\dots,a_m \in D^*$, the value \(f(a_1,\dots,a_m)\) is algebraic over \(F\) of degree at most \(d\).  Then $\dim_F D \;\le\; d^2$.
\end{theorem}

\begin{proof}
By Lemma \ref{lem:auxiliary}, \(D\) is finite‑dimensional over \(F\).  Write \(\dim_F D = n^2\) for some \(n\in\mathbb{N}\).  Suppose, towards a contradiction, that \(n>d\).  Then Theorems \ref{MaxSbFi_poly} and \ref{MaxSbFi_word} guarantee the existence of \(a_1,\dots,a_m\in D^*\) for which
$K \;:=\; F\bigl(f(a_1,\dots,a_m)\bigr)$
is a maximal subfield of \(D\).  By Lemma \ref{lem3.3_Aaghabali.Bien.2019paper}, 
\(\dim_F K \;=\; n^2,\)
so \(f(a_1,\dots,a_m)\) is algebraic of degree \(n\) over \(F\).  This contradicts the bound \(n\le d\).  Hence \(\dim_F D\le d^2\).
\end{proof}

\section*{Declarations}

\begin{itemize}
\item \textbf{Ethical Approval:} The authors declare that they have no competing interests relevant to the content of this article.
\item \textbf{Competing Interests:} The authors declare no conflicts of interest.
\item \textbf{Authors' Contributions:} Both authors contributed equally to the research and manuscript preparation. The corresponding author, Le Qui Danh, carried out the final editing, which was approved by the co‑author.
\item \textbf{Availability of Data and Materials:} Not applicable; no datasets or other materials were generated or analyzed in the course of this study.
\end{itemize}

\end{document}